\documentclass[11pt,reqno]{amsart}
\usepackage{fixltx2e}                    
\usepackage{mathpazo}  
\usepackage[utf8]{inputenc}            
\usepackage{amsmath}                     
\usepackage{amssymb, latexsym, stmaryrd, amsthm, dsfont, amsfonts, amsbsy,amsthm, amsmath, mathrsfs}            
\usepackage{mathtools}                   
\usepackage{bm}                          
\usepackage{enumerate}                   
\usepackage{verbatim}                    
\usepackage{url}   
\usepackage{lscape}                      
\usepackage{soul}
\usepackage{centernot}
\usepackage{upgreek}

\usepackage{bussproofs} 

\usepackage[margin=0.8in]{geometry}

\usepackage[mathscr]{euscript}
 \let\mathscr\relax
\usepackage[scr]{rsfso}

\usepackage{microtype,tikz,tikz-cd}  
\makeatletter                            
\def\MT@register@subst@font{\MT@exp@one@n\MT@in@clist\font@name\MT@font@list
 \ifMT@inlist@\else\xdef\MT@font@list{\MT@font@list\font@name,}\fi}
\makeatother

\usepackage[pdftex,bookmarks,bookmarksnumbered,linktocpage,   
         colorlinks,linkcolor=blue,citecolor=blue]{hyperref}

\newcommand{\bit}{\begin{itemize}}    
\newcommand{\eit}{\end{itemize}}
\newcommand{\ben}{\begin{enumerate}}
\newcommand{\een}{\end{enumerate}}

\newcommand{\benroman}{\ben[\normalfont (i)]}  
\let\eroman\een

\newcommand{\bde}{\begin{description}}
\newcommand{\ede}{\end{description}}
\let\oper=\mathbb                               

\newcommand{\III}{\oper{I}}                     
\newcommand{\SSS}{\oper{S}}

\newcommand{\VVV}{\oper{V}}                     
\newcommand{\QQQ}{\oper{Q}}

\newcommand{\HHH}{\oper{H}}

\newcommand{\PPP}{\oper{P}}

\newcommand{\PPU}{\oper{P}_{\!\textsc{\textup{u}}}^{}}

\theoremstyle{theorem}
\newtheorem{Theorem}{Theorem}[section]
\newtheorem{Theorem-n}{Theorem}
\newtheorem{Embedding Lemma}[Theorem]{Embedding Lemma}

\newtheorem{Trichotomy Theorem}[Theorem]{Antidichotomy Theorem}
\newtheorem{Correspondence Theorem}[Theorem]{Correspondence Theorem}
\newtheorem{Subdirect Decomposition Theorem}[Theorem]{Subdirect Decomposition Theorem}

\newtheorem{Proposition}[Theorem]{Proposition}

\newtheorem{Lemma}[Theorem]{Lemma}
\newtheorem{Corollary}[Theorem]{Corollary}

\theoremstyle{definition}
\newtheorem{Definition}[Theorem]{Definition}
\newtheorem{exa}[Theorem]{Example}

\theoremstyle{remark}

\newtheorem{Remark}[Theorem]{Remark}

\let\leq=\leqslant

\let\geq=\geqslant 

\bmdefine{\2}{2}
\bmdefine{\A}{A} 
\bmdefine{\B}{B}
\bmdefine{\C}{C}
\bmdefine{\D}{D}

\usepackage{scalerel}

\makeatletter
\@namedef{subjclassname@2020}{\textup{2020} Mathematics Subject Classification}
\makeatother



\begin{document}

\title[Epimorphisms between finitely generated algebras]{Epimorphisms between finitely generated algebras}

\author{Luca Carai, Miriam Kurtzhals, and Tommaso Moraschini}

\address{Luca Carai: Dipartimento di Matematica ``Federigo Enriques'', Universit\`a degli Studi di Milano, via Cesare Saldini 50, 20133 Milan, Italy}\email{luca.carai.uni@gmail.com}

\address{Miriam Kurtzhals and Tommaso Moraschini: Departament de Filosofia, Facultat de Filosofia, Universitat de Barcelona (UB), Carrer Montalegre, $6$, $08001$ Barcelona, Spain}
\email{mkurtzku7@alumnes.ub.edu \textnormal{and} tommaso.moraschini@ub.edu}
\email{ }

\date{\today}

\begin{abstract}
A quasivariety has the \emph{weak ES property} when the epimorphisms between its finitely generated members are surjective. A characterization of quasivarieties with the weak ES property is obtained and a method for detecting failures of this property in quasivarieties with a near unanimity term and in congruence permutable varieties is given. It is also shown that under reasonable assumptions the weak ES property implies arithmeticity. In particular, every filtral variety with the weak ES property is a discriminator variety.
\end{abstract}

\maketitle

\section{Introduction}

Let $\mathsf{K}$ be a class of algebras. A homomorphism $f \colon \A \to \B$ with $\A, \B \in \mathsf{K}$ is said to be a $\mathsf{K}$-\emph{epimorphism} when it is right cancellable, that is, when for every pair of homomorphisms $g, h \colon \B \to \C$ with $\C \in \mathsf{K}$,
\[
g \circ f = h \circ f\, \, \text{ implies }\, \, g=h.
\]
While every surjective homomorphism between members of $\mathsf{K}$ is a $\mathsf{K}$-epimorphism, the converse need not hold in general. For instance, the inclusion map of the integers into the rationals is a nonsurjective epimorphism in the variety of rings (see, e.g., \cite{Isbell66epi}). 

The demand that all $\mathsf{K}$-epimorphisms be surjective has been studied extensively (see \cite{KMPT83} and the references therein). In this paper, we focus on the strictly weaker demand that all $\mathsf{K}$-epimorphisms between \emph{finitely generated} members of $\mathsf{K}$ be surjective, in which case $\mathsf{K}$ is said to have the  \emph{weak epimorphism surjectivity property} (weak ES property, for short) \cite[p.\ 259]{HMTII}. Furthermore, we restrict our attention to the case where $\mathsf{K}$ is a \emph{variety} or a \emph{quasivariety}, that is, a class of algebras axiomatized by equations or by quasiequations
\cite{BuSa00,Go98a}.

The interest of the weak ES property is twofold. On the one hand, it is often the case that quasivarieties in which epimorphisms need not be surjective lack the weak ES property as well.\ For instance, the inclusion map of $\oper{Z}$ into $\oper{Z}[1/ p]$ for a prime $p$ is a nonsurjective epimorphism between finitely generated rings, whence the variety of rings lacks the weak ES property. More in general, the ES property and its weak version coincide for quasivarieties with the amalgamation property (see, e.g., \cite[Thm.\ 1.3]{BMR16}). On the other hand, from a logical standpoint, the weak ES property is the algebraic counterpart of the \emph{(finite) Beth definability property}, i.e., the demand that all implicit definitions can be made explicit (see \cite[Thm.\ 5.6.10]{HMTII} and \cite[Thm.\ 3.14]{BlHoo06}). 
In essence, the latter means that if an element of a structure satisfying some property is unique when it exists, then it is definable by a term of the language.
For instance, the failure of the weak ES property in rings amounts to the fact that multiplicative inverses are implicitly, but not explicitly, definable.  

In general, the task of determining whether a quasivariety $\mathsf{K}$ has the weak ES property is nontrivial, in part because of the difficulty of checking which homomorphisms between \emph{arbitrary} finitely generated members of $\mathsf{K}$ are $\mathsf{K}$-epimorphisms. Our main results simplify this task in different ways. 

On the one hand, we provide a characterization of the quasivarieties with the weak ES property (Theorem \ref{Thm : weak ES vs fully epic}) and apply it to show that, under reasonable assumptions, a quasivariety $\mathsf{K}$ has the weak ES property provided that no proper inclusion map $f \colon \A \to \B$, where $\B$ is a finitely generated member of $\mathsf{K}$ ``of a specific kind'', is a $\mathsf{K}$-epimorphism. More precisely, we recall that 
$\mathsf{K}$ has a \emph{near unanimity term} \cite{KaaPix}  if there exists a term $\varphi$ of arity $\geq 3$ such that
\[
\mathsf{K} \vDash x \thickapprox \varphi(y, x, \dots, x) \thickapprox \varphi(x, y, x, \dots, x) \thickapprox \dots \thickapprox \varphi(x, \dots, x, y).
\]
We show that if $\mathsf{K}$ has a near unanimity term of arity $n$, we may assume that the algebra $\B$ is a subdirect product $\B \leq \C_1 \times \dots \times \C_{n-1}$ of  relatively finitely subdirectly irreducible factors (Theorem \ref{Thm : near unanimity}). 
Furthermore, we prove that if $\mathsf{K}$ is a congruence permutable variety, we may  assume that the algebra $\B$ is finitely subdirectly irreducible (Theorem \ref{Thm : CP locally finite}). Notably, the first result applies to every variety of algebras with a lattice reduct and the second to every variety of algebras with a group reduct (see, e.g., \cite[p.\ 79 \& Thm.\ II.12.2]{BuSa00}), thus covering most examples from algebra and logic.\footnote{We remark that Theorems \ref{Thm : near unanimity} and \ref{Thm : CP locally finite} are slightly stronger than suggested here, but their formulation would require additional definitions.}

On the other hand, we show that, although the theory of the weak ES property is simpler than that of its traditional variant, the weak ES property has a profound impact on the structure theory of quasivarieties.
More precisely, we show that for congruence distributive quasivarieties  $\mathsf{K}$ whose class of relatively finitely subdirectly irreducible members is closed under nontrivial subalgebras, the weak ES property implies that the variety generated by $\mathsf{K}$ is \emph{arithmetical}, i.e., both congruence distributive and congruence permutable (Theorem \ref{Thm : weak ES implies CP}). As a consequence, every \emph{filtral variety}  \cite{Magari1969} with the weak ES property must be a \emph{discriminator variety} \cite{We78} (Example \ref{Cor : final corrolary}), see also \cite{CamperVagg}.

We remark that Theorems \ref{Thm : near unanimity} and \ref{Thm : CP locally finite} strengthen similar observations on the surjectivity of all $\mathsf{K}$-epimorphisms \cite[Thms.\ 6.4 \& 6.8]{Camper18jsl} (here Theorems \ref{Thm : Campercholi unanimity} and \ref{Thm : Campercholi CP}). Both our improvements require new proof strategies. For instance, \cite[Thm.\ 6.8]{Camper18jsl} states that all $\mathsf{K}$-epimorphisms are surjective for an arithmetical variety $\mathsf{K}$ whose class of finitely subdirectly irreducible members is universal provided that no proper inclusion map $f \colon \A \to \B$, where $\B$ is  finitely subdirectly irreducible, is a $\mathsf{K}$-epimorphism.  We show that, for the weak ES property, this is true under the sole assumption of congruence permutability (Theorem \ref{Thm : CP locally finite}).  This allows us to dispense with the main ingredients of \cite{Camper18jsl}, namely, the theory of definability in \cite{CamperVagg-DEF},  sheaf representations \cite{GramaVaggio96,KrCl79}, and the \emph{infinitary Baker-Pixley Theorem} \cite{VaggBP18} (see also \cite{CamperVagg-BaPix}). 
Instead, we capitalize on the new notion of a full subalgebra and its interaction with subdirect representations.

\section{Varieties and quasivarieties}

The classes of algebras considered in this paper will be assumed to comprise only similar algebras. We denote the class operators of closure under isomorphic copies, subalgebras, homomorphic images, direct products, and ultraproducts by $\III, \SSS, \HHH, \PPP$, and $\PPU$, respectively. A class of algebras is said to be:
\benroman
\item a \emph{variety} when it is closed under $\HHH, \SSS$, and $\PPP$;
\item a \emph{quasivariety} when it is closed under $\III, \SSS, \PPP$, and $\PPU$.
\eroman
In view of Birkhoff's and Maltsev's Theorems, varieties and quasivarieties coincide with the classes of algebras axiomatized by equations and quasiequations, respectively (see, e.g., \cite[Thms.\ II.11.9 \& V.2.25]{BuSa00}). While every variety is a quasivariety, the converse is not true in general.\ We denote the least variety and the least quasivariety containing a class of algebras $\mathsf{K}$ by $\VVV(\mathsf{K})$ and $\QQQ(\mathsf{K})$, respectively. A variety (resp.\ quasivariety) $\mathsf{K}$ is \emph{finitely generated} when $\mathsf{K} = \VVV(\mathsf{M})$ (resp.\ $\mathsf{K} = \QQQ(\mathsf{M})$) for a finite set $\mathsf{M}$ of finite algebras.

As quasivarieties need not be closed under $\HHH$, the following concept is often useful.\ Let $\mathsf{K}$ be a quasivariety and $\A \in \mathsf{K}$. A congruence $\theta$ of $\A$ is said to be a $\mathsf{K}$\emph{-congruence} when $\A / \theta \in \mathsf{K}$. Owing to the fact that $\mathsf{K}$ is closed under $\III$ and $\SSS$,  the Homomorphism Theorem yields that  the kernel
\[
\mathsf{Ker}(f) \coloneqq \{ \langle a, b \rangle \in A \times A : f(a) = f(b) \}
\]
of every homomorphism $f \colon \A \to \B$ with $\B \in \mathsf{K}$ is a $\mathsf{K}$-congruence of $\A$.\ When ordered under the inclusion relation, the set of $\mathsf{K}$-congruences of $\A$ forms an algebraic lattice $\mathsf{Con}_\mathsf{K}(\A)$ in which  meets are intersections. Given $X \subseteq A \times A$, we denote the least $\mathsf{K}$-congruence of $\A$ containing $X$ by $\textup{Cg}_\mathsf{K}^\A(X)$. When $\mathsf{K}$ is a variety, $\mathsf{Con}_\mathsf{K}(\A)$ coincides with the congruence lattice $\mathsf{Con}(\A)$ of $\A$ and $\textup{Cg}_\mathsf{K}^\A(X)$ is the least congruence of $\A$ containing $X$, in symbols, $\textup{Cg}^\A(X)$. 
An immediate generalization to quasivarieties of \cite[Thm.\ II.6.20]{BuSa00} shows that, given a member $\A$ of a quasivariety $\mathsf{K}$ and $\theta \in \mathsf{Con}_\mathsf{K}(\A)$, the lattice $\mathsf{Con}_\mathsf{K}(\A / \theta)$ can be described as follows:

\begin{Correspondence Theorem}\label{Thm : correspondence}
Let $\mathsf{K}$ be a quasivariety, $\A \in \mathsf{K}$, and $\theta \in \mathsf{Con}_\mathsf{K}(\A)$. Moreover, let ${\uparrow}\theta$ be the sublattice of $\mathsf{Con}_\mathsf{K}(\A)$ with universe $\{ \phi \in \mathsf{Con}_\mathsf{K}(\A) : \theta \subseteq \phi \}$. Then the map
\[
f \colon {\uparrow}\theta \to \mathsf{Con}_\mathsf{K}(\A / \theta)
\]
defined by the rule $f(\phi) \coloneqq \{ \langle a / \theta, b / \theta \rangle : \langle a, b \rangle \in \phi \}$ is a lattice isomorphism.
\end{Correspondence Theorem}

Given a homomorphism $f \colon \A \to \B$, we denote the subalgebra of $\B$ with universe $f[A]$ by $f[\A]$. Furthermore, we write $\A \leq \B$ to indicate that $\A$ is a subalgebra of $\B$. Then an algebra $\A$ is a \emph{subdirect product} of a family $\{ \B_i : i \in I \}$ when $\A \leq \prod_{i \in I}\B_i$ and for every $i \in I$ the projection map $p_i \colon \A \to \B_i$ is surjective.\ Similarly, an embedding $f \colon \A \to \prod_{i \in I}\B_i$ is called \emph{subdirect} when $f[\A]\leq \prod_{i \in I}\B_i$ is a subdirect product. The next result simplifies the task of constructing subdirect embeddings (see, e.g., \cite[Lem.\ II.8.2]{BuSa00}): 

\begin{Proposition}\label{Prop : subdirect embedding}
Let $\A$ be an algebra and $X \subseteq \mathsf{Con}(\A)$. Then the map
\[
f \colon \A / \bigcap X \to \prod_{\theta \in X}\A / \theta 
\]
defined by the rule $f(a / \bigcap X) \coloneqq \langle a / \theta : \theta \in X \rangle$ is a subdirect embedding.
\end{Proposition}

Let $\mathsf{K}$ be a quasivariety. An algebra $\A \in \mathsf{K}$ is said to be \emph{subdirectly irreducible} relative to $\mathsf{K}$ (RSI, for short) when for every subdirect embedding $f \colon \A \to \prod_{i \in I}\B_i$ with $\{ \B_i : i \in I \} \subseteq \mathsf{K}$ there exists $i \in I$ such that $p_i \circ f \colon \A \to \B_i$ is an isomorphism. In case this happens whenever the index set $I$ is finite, 
we say that $\A$ is \emph{finitely subdirectly irreducible} relative to $\mathsf{K}$ (RFSI, for short).\footnote{ When the language under consideration is clear, we adopt the following convention: the direct product of an empty family of algebras is the trivial algebra in the language under consideration. Consequently, trivial algebras are considered neither RSI nor RFSI.} 
The classes of RSI and RFSI members of $\mathsf{K}$ will be denoted by $\mathsf{K}_{\textsc{rsi}}$ and $\mathsf{K}_{\textsc{rfsi}}$, respectively. When $\mathsf{K}$ is a variety, the requirement that $\{ \B_i : i \in I \}$ is a subset of $\mathsf{K}$ in the above definitions can be harmlessly dropped and we simply say that $\A$ is \emph{subdirectly irreducible} (SI) or \emph{finitely subdirectly irreducible} (FSI) (i.e., we drop the ``relative to  $\mathsf{K}$''). In this case, we also write $\mathsf{K}_{\textsc{si}}$ and $\mathsf{K}_{\textsc{fsi}}$ instead of $\mathsf{K}_{\textsc{rsi}}$ and $\mathsf{K}_{\textsc{rfsi}}$.

The importance of subdirect embeddings and R(F)SI algebras derives from the following  representation theorem (see, e.g., \cite[Thm.\ 3.1.1]{Go98a}):

\begin{Subdirect Decomposition Theorem}\label{Thm : Subdirect Decomposition}
Let $\mathsf{K}$ be a quasivariety. For every $\A \in \mathsf{K}$ there exists a subdirect embedding $f \colon \A \to \prod_{i \in I}\B_i$ with $\{ \B_i : i \in I \} \subseteq \mathsf{K}_{\textsc{rsi}}$. When $\A$ is finite, $I$ and each $\B_i$ can be chosen finite. 
\end{Subdirect Decomposition Theorem}

Notably, the RSI and RFSI members of a quasivariety $\mathsf{K}$ can be recognized by looking at the structure of their lattices of $\mathsf{K}$-congruences.
To explain how, we recall that an element $a$ of a complete lattice $\A$ is

\benroman
\item  \emph{completely meet irreducible} if for every family $\{a_i : i \in I\} \subseteq A$,
    \[
    a = \bigwedge_{i \in I} a_i \textrm{ implies } a = a_i \textrm{ for some } i \in I;
    \]
    \item \emph{meet irreducible} if the above display holds whenever $I$ is finite.
\eroman
 As the set of indices $I$ is allowed to be empty, the greatest element of $\A$ is never meet irreducible.

Given $\A \in \mathsf{K}$, let
\begin{align*}
\mathsf{Irr}^\infty_\mathsf{K}(\A) &\coloneqq \text{the set of completely meet irreducible elements of }\mathsf{Con}_\mathsf{K}(\A);\\
\mathsf{Irr}_{\mathsf{K}}(\A) &\coloneqq \text{the set of meet irreducible elements of }\mathsf{Con}_\mathsf{K}(\A).
\end{align*}
Furthermore, we denote the identity relation on $\A$ by $\textup{id}_A$.\  The following is a consequence of the Correspondence Theorem \ref{Thm : correspondence} and \cite[Cor.\ 1.4.8]{Go98a}:

\begin{Proposition}\label{Prop : RFSI}
Let $\A$ be a member of a quasivariety $\mathsf{K}$. For every $\theta \in \mathsf{Con}(\A)$ we have
\begin{align*}
\A / \theta \in \mathsf{K}_{\textup{\textsc{rsi}}} \, \, &\text{ if and only if } \, \, \theta \in \mathsf{Irr}_{\mathsf{K}}^\infty(\A);\\
\A / \theta \in \mathsf{K}_{\textup{\textsc{rfsi}}} \, \, &\text{ if and only if } \, \, \theta \in \mathsf{Irr}_{\mathsf{K}}(\A).
\end{align*}
Therefore, $\A\in \mathsf{K}_{\textup{\textsc{rsi}}}$  \textup{(}resp.\ $\A\in \mathsf{K}_{\textup{\textsc{rfsi}}}$\textup{)} if and only if $\textup{id}_A \in \mathsf{Irr}_{\mathsf{K}}^\infty(\A)$ \textup{(}resp.\ $\textup{id}_A \in \mathsf{Irr}_{\mathsf{K}}(\A)$\textup{)}.
\end{Proposition}

As a consequence, a member $\A$ of a quasivariety $\mathsf{K}$ is RSI precisely when it has a least nonidentity $\mathsf{K}$-congruence, called the \emph{monolith} of $\A$. When it exists, the monolith of $\A$ is always the $\mathsf{K}$-congruence of $\A$ generated by a pair of distinct elements $a, b \in A$, which we denote by $\textup{Cg}_\mathsf{K}^\A(a, b)$.

 Let $\A_1, \dots, \A_n$ be algebras and $\theta_i \in \mathsf{Con}(\A_i)$ for each $i \leq n$. Then the relation
\[
\theta_1 \times \dots \times \theta_n \coloneqq \{ \langle \langle a_1, \dots, a_n \rangle, \langle b_1, \dots, b_n \rangle \rangle \in (A_1\times \dots \times A_n)^2 : \langle a_i, b_i \rangle \in \theta_i \text{ for each  } i \leq n \}
\]
is a congruence of the direct product $\A_1 \times \dots \times \A_n$. Given a pair of algebras $\A \leq \B$ and $\theta \in \mathsf{Con}(\B)$, we write $\theta{\upharpoonright}_A$ as a shorthand for $\theta \cap (A \times A)$. Notice that $\theta{\upharpoonright}_A$ is a congruence of $\A$. 
A congruence $\theta$ of a subdirect product $\A \leq \B_1 \times \dots \times \B_n$ is said to be a \emph{product congruence} when $\theta = (\theta_1 \times \dots \times \theta_n){\upharpoonright}_A$ for some $\theta_1 \in \mathsf{Con}(\B_1), \dots, \theta_n \in \mathsf{Con}(\B_n)$. 

A quasivariety $\mathsf{K}$ is said to be \emph{congruence distributive} when $\mathsf{Con}_\mathsf{K}(\A)$ is a distributive lattice for every $\A \in \mathsf{K}$.  The next result is an effortless generalization to quasivarieties of \cite[Thm.\ 1.2.20]{KaaPix}:
\begin{Theorem}\label{Thm : CD varieties}
A quasivariety $\mathsf{K}$ is congruence distributive iff for every subdirect product $\A \leq \B_1 \times \dots \times  \B_n$ with $\B_1, \dots, \B_n \in \mathsf{K}$ and every $\theta \in \mathsf{Con}_\mathsf{K}(\A)$ there exist $\theta_1 \in \mathsf{Con}_\mathsf{K}(\B_1), \dots, \theta_n \in \mathsf{Con}_\mathsf{K}(\B_n)$ such that $\theta = (\theta_1 \times \dots \times \theta_n){\upharpoonright}_A$.
\end{Theorem}

The following theorem can be derived from
J\'onsson's Lemma (see, e.g., \cite[Lem.\ 5.9]{Be11g}), which is one of the main consequences of congruence distributivity. 
\begin{Theorem}\label{Jonsson Lemma}
Let $\mathsf{K}$ be a class of algebras such that $\VVV(\mathsf{K})$ is congruence distributive. Then the FSI members of $\VVV(\mathsf{K})$ belong to $\HHH\SSS\PPU(\mathsf{K})$.
\end{Theorem}

Lastly, given an algebra $\A$ and a set $X \subseteq A$, we denote the least subuniverse of $\A$ containing $X$ by $\textup{Sg}^\A(X)$. When $A = \textup{Sg}^\A(X)$ for some finite $X \subseteq A$, we say that $\A$ is \emph{finitely generated}. If every finitely generated subalgebra of $\A$ is finite, we call $\A$ \emph{locally finite}. A class of algebras is \emph{locally finite} when its members are.

\section{Epimorphism surjectivity}

\begin{Definition}
A quasivariety $\mathsf{K}$ is said to have:
\benroman
\item the \emph{epimorphism surjectivity property} (ES property, for short) when every $\mathsf{K}$-epimorphism is surjective;
\item the \emph{weak epimorphism surjectivity property} (weak ES property, for short) when every $\mathsf{K}$-epimorphism between finitely generated algebras is surjective.\footnote{
The weak ES property is often phrased as the demand that every $\mathsf{K}$-epimorphism $f \colon \A \to \B$ such that $B = \textup{Sg}^\B(f[A] \cup X)$ for some finite $X \subseteq B$ is surjective \cite{BlHoo06}. From Proposition \ref{Prop : epic} and \cite[Thm.\ 5.4]{MRJ18sl} it follows that the two definitions are equivalent.}
\eroman
\end{Definition}

Examples separating the ES property from its weak variant abound.\ For instance, when phrased in algebraic terms, a theorem of Kreisel \cite[Thm.\ 1]{Kreisel60JSL} states that all varieties of Heyting algebras have the weak ES property  (for an algebraic proof, see Example \ref{Exa : Kreisel}). However, a continuum of them lacks the ES property \cite[Thm.\ 8.4]{MorWan19es} (see also \cite{BMR16}). 

The task of determining whether a quasivariety has the weak ES property can be simplified 
using the notion of an epic subalgebra.
Given a quasivariety $\mathsf{K}$ and $\B \in \mathsf{K}$, we say that a subalgebra $\A \leq \B$ is \emph{epic} in $\mathsf{K}$ when the inclusion map $i \colon \A \to \B$ is a $\mathsf{K}$-epimorphism, that is, when for every $\C \in \mathsf{K}$ and every pair of homomorphisms $g, h \colon \B \to \C$,
\[
g{\upharpoonright}_A = h{\upharpoonright}_A \, \, \text{ implies } \, \, g= h. 
\]
Notice that a homomorphism $f \colon \A \to \B$ between members of $\mathsf{K}$ is a $\mathsf{K}$-epimorphism iff $f[\A] \leq \B$ is epic in $\mathsf{K}$. Lastly, we say that $\A \leq \B$ is \emph{almost total} when there exists some $b \in B$ such that $B = \textup{Sg}^\B(A \cup \{ b \})$. We will prove that these concepts are related as follows:

\begin{Proposition}\label{Prop : epic}
 A quasivariety $\mathsf{K}$ has the weak ES property iff its finitely generated members lack proper subalgebras that are almost total and epic in $\mathsf{K}$.
\end{Proposition}

Since the subalgebras of a finitely generated algebra need not be finitely generated, the next observation is required to prove the implication from left to right in Proposition \ref{Prop : epic}:

\begin{Lemma}\label{Lem : from thm 5.4}
Let $\mathsf{K}$ be a quasivariety, $\B \in \mathsf{K}$, and $\A \leq \B$ proper, almost total, and epic in $\mathsf{K}$. Then there exist a finitely generated $\B' \in \mathsf{K}$ and $\A' \leq \B'$ finitely generated, proper, almost total, and epic in $\mathsf{K}$.
\end{Lemma}

\begin{proof}
This is established in the proof of \cite[Thm.\ 5.4]{MRJ18sl} (see also \cite{Bacsich47}).
\end{proof}

We are now ready to prove Proposition \ref{Prop : epic}.

\begin{proof}[Proof of Proposition \ref{Prop : epic}.]
The implication from left to right follows from Lemma \ref{Lem : from thm 5.4}. To prove the other implication, suppose that $\mathsf{K}$ lacks the weak ES property. Then there exists a nonsurjective $\mathsf{K}$-epimorphism $f \colon \A \to \B$ with $\A$ and $\B$ finitely generated. Therefore, $f[\A] \leq \B$ is proper and epic in $\mathsf{K}$. Now, let $G$ be a finite set of generators for $\B$. As $f[\A] \leq \B$ is proper, the set $G - f[A]$ is nonempty. Then there are $X \subseteq G - f[A]$ and $b \in G-f[A]$ such that
\[
b \notin \textup{Sg}^\B(f[A] \cup X) \, \, \text{ and } \, \, B = \textup{Sg}^\B(f[A] \cup X \cup \{ b\}).
\]
Let $\C$ be the subalgebra of $\B$ with universe $\textup{Sg}^\B(f[A] \cup X)$. In view of the the above display, 
\[
b \notin C \, \, \text{ and } \, \, B = \textup{Sg}^\B(C \cup \{ b \}).
\]
Thus, $\C \leq \B$ is proper and almost total.\ As $f[\A] \leq \C$ by the construction of $\C$, from the assumption that $f[\A] \leq \B$ is epic in $\mathsf{K}$ it follows that $\C \leq \B$ is also epic in $\mathsf{K}$.
\end{proof}

The rest of this section is devoted to improving Proposition \ref{Prop : epic}. The next concept is instrumental to this purpose:

\begin{Definition}
Let $\B$ be a member of a quasivariety $\mathsf{K}$. A subalgebra $\A \leq \B$ is \emph{full} in $\mathsf{K}$ when it is proper, almost total, and for every $\theta \in \mathsf{Con}_\mathsf{K}(\B)$ it holds that
\[
\text{if }\theta \ne \textup{id}_B, \text{ then for every }b \in B \text{ there exists }a \in A \text{ s.t. }\langle a, b \rangle \in \theta.
\]
When $\A \leq \B$ is both epic and full in $\mathsf{K}$, we say that $\A \leq \B$ is \emph{fully epic} in $\mathsf{K}$.
\end{Definition}

Our aim is to prove the following:

\begin{Theorem}\label{Thm : weak ES vs fully epic}
A quasivariety $\mathsf{K}$ has the weak ES property iff for every finitely generated $\B \in \mathsf{K}$ and $\A \leq \B$ that is full in $\mathsf{K}$ one of the following conditions holds:
\benroman
\item\label{weak ES vs fully epic : item : 1} There are two distinct $\theta, \phi \in \mathsf{Con}_\mathsf{K}(\B)$ such that $\theta{\upharpoonright}_A = \phi{\upharpoonright}_A$;
\item\label{weak ES vs fully epic : item : 2} There are two distinct embeddings $g, h \colon \B \to \C$ with $\C \in \mathsf{K}_{\textup{\textsc{rsi}}}$ such that $g{\upharpoonright}_A = h{\upharpoonright}_A$.
\eroman 
Moreover, if condition (\ref{weak ES vs fully epic : item : 1}) holds, we may assume that $\theta = \textup{id}_B$.
\end{Theorem}

The proof of Theorem \ref{Thm : weak ES vs fully epic} proceeds through a series of technical observations. Given an algebra $\A$, we say that a family $\{ \theta_i : i \in I \} \subseteq \mathsf{Con}(\A)$ is a \emph{chain} when for every $i, j \in I$ either $\theta_i \subseteq \theta_j$ or $\theta_j \subseteq \theta_i$. We will require the following well-known fact.

\begin{Proposition}\label{Prop : chains of congruences}
Let $\A$ be a member of a quasivariety $\mathsf{K}$. The union of a nonempty chain of $\mathsf{K}$-congruences of $\A$ is still a $\mathsf{K}$-congruence of $\A$.
\end{Proposition}

When $\A \leq \B$ and $\theta \in \mathsf{Con}_\mathsf{K}(\B)$, we denote the subalgebra of $\B / \theta$ with universe $\{ a / \theta : a \in A \}$ by $\A / \theta$.

\begin{Proposition}\label{Prop : full quotients}
Let $\mathsf{K}$ be a quasivariety, $\B \in \mathsf{K}$, and $\A \leq \B$ proper and almost total. Then there exists $\theta \in \mathsf{Con}_\mathsf{K}(\B)$ such that $\A / \theta \leq \B / \theta$ is full in $\mathsf{K}$.
\end{Proposition}

\begin{proof}
As $\A \leq \B$ is proper and almost total, there exists $b \in B - A$ such that $B = \textup{Sg}^\B(A \cup \{ b \})$. Then consider the poset
\[
X \coloneqq \{ \theta \in \mathsf{Con}_\mathsf{K}(\B) : \text{there exists no }a \in A \text{ s.t. } \langle a, b \rangle \in \theta \}
\]
ordered under the inclusion relation. We will apply Zorn's Lemma to obtain a maximal element of $X$. Clearly, $X$ contains $\textup{id}_B$. Furthermore, the definition of $X$ and Proposition \ref{Prop : chains of congruences} guarantee that $X$ is closed under unions of nonempty chains. Therefore, there exists a maximal element $\theta$ of $X$.

From $\theta \in X$ it follows that $b / \theta$ does not belong to $A / \theta$. Therefore, $\A / \theta \leq \B / \theta$ is proper. Moreover, $B = \textup{Sg}^\B(A \cup \{ b \})$ implies that $B / \theta = \textup{Sg}^{\B / \theta}(A / \theta \cup \{ b / \theta \})$. Therefore, $\A / \theta \leq \B / \theta$ is almost total. It only remains to prove that it is full in $\mathsf{K}$.

To this end, consider some $\phi \in \mathsf{Con}_\mathsf{K}(\B / \theta) - \{ \textup{id}_{B/\theta} \}$. By the Correspondence Theorem \ref{Thm : correspondence} there exists $\eta \in \mathsf{Con}_\mathsf{K}(\B)$ such that
\[
\theta \subset \eta \, \, \text{ and } \, \, \phi = \{ \langle a / \theta, c / \theta \rangle : \langle a, c \rangle \in \eta \}.
\] 
Since $\theta$ is a maximal element of $X$, from $\theta \subset \eta$ it follows that $\eta \notin X$.\ Therefore, there exists  $a \in A$ such that $\langle a, b \rangle \in \eta$. In view of the above display, this yields $\langle a/ \theta, b / \theta \rangle \in \phi$. Then consider an arbitrary $c / \theta \in B / \theta$. In order to conclude the proof, we need to find some $c_a / \theta \in A / \theta$ such that $\langle c/ \theta, c_a / \theta \rangle \in \phi$. Since $B = \textup{Sg}^{\B}(A \cup \{ b \})$, there exist a term $\varphi(x_1, \dots, x_n, y)$ and $a_1, \dots, a_n \in A$ such that $c = \varphi^{\B}(a_1, \dots, a_n, b)$. Together with $\langle a/ \theta, b / \theta \rangle \in \phi$, this implies
\[
\langle c / \theta, \varphi^{\B/\theta}(a_1/ \theta, \dots, a_n / \theta, a / \theta) \rangle \in \phi.
\]
As $a, a_1, \dots, a_n \in A$, the element $c_a \coloneqq \varphi^{\B}(a_1, \dots, a_n, a)$ belongs to $A$ and, therefore, $c_a / \theta \in A / \theta$. Lastly, the above display amounts to $\langle c / \theta, c_a / \theta \rangle\in \phi$.
\end{proof}

\begin{Corollary}\label{Cor : fully epic = weak ES}
A quasivariety $\mathsf{K}$ has the weak ES property iff its finitely generated members lack subalgebras that are fully epic in $\mathsf{K}$.
\end{Corollary}

\begin{proof}
The implication from left to right follows from Proposition \ref{Prop : epic}. To prove the other implication, suppose by contraposition that $\mathsf{K}$ lacks the weak ES property. By Proposition \ref{Prop : epic} there exist $\B \in \mathsf{K}$ finitely generated and $\A \leq \B$ proper, almost total, and epic in $\mathsf{K}$. Therefore, we can apply Proposition \ref{Prop : full quotients} obtaining  $\theta \in \mathsf{Con}_\mathsf{K}(\B)$ such that $\A / \theta \leq \B / \theta$ is full in $\mathsf{K}$. Furthermore, from the assumption that $\A \leq \B$ is epic in $\mathsf{K}$ it follows that so is $\A / \theta \leq \B / \theta$. 
\end{proof}

We will also make use of the next technical observation:

\begin{Lemma}\label{Lem : technical phi theta separation}
Let $\mathsf{K}$ be a quasivariety, $\B \in \mathsf{K}$, and $\A \leq \B$ full in $\mathsf{K}$.  If $\theta, \phi \in \mathsf{Con}_\mathsf{K}(\B)$ are such that $\theta \ne \phi$ and $\theta{\upharpoonright}_A = \phi{\upharpoonright}_A$, then $\theta \cap \phi = \textup{id}_B$.
\end{Lemma}

\begin{proof}
Assume that $\theta, \phi \in \mathsf{Con}_\mathsf{K}(\B)$ are such that $\theta \ne \phi$ and $\theta{\upharpoonright}_A = \phi{\upharpoonright}_A$.\ Suppose, with a view to contradiction, that $\theta \cap \phi \ne \textup{id}_B$. Since $\A \leq \B$ is full in $\mathsf{K}$, for every $b_1,b_2 \in B$ there exist $a_1,a_2 \in A$ such that $\langle a_1, b_1 \rangle, \langle a_2, b_2 \rangle \in \theta \cap \phi$. Then $\langle b_1, b_2 \rangle \in \theta$ if and only if $\langle a_1, a_2 \rangle \in \theta$, and  $\langle b_1, b_2 \rangle \in \phi$ if and only if $\langle a_1, a_2 \rangle \in \phi$. From $\theta{\upharpoonright}_A = \phi{\upharpoonright}_A$ it follows that $\langle a_1, a_2 \rangle \in \theta$ if and only if $\langle a_1, a_2 \rangle \in \phi$. Thus, $\langle b_1, b_2 \rangle \in \theta$ if and only if $\langle b_1, b_2 \rangle \in \phi$. But this implies $\theta = \phi$, which contradicts our assumption.
\end{proof}

The heart of the proof of Theorem \ref{Thm : weak ES vs fully epic} is the next observation:

\begin{Proposition}\label{Prop : weak ES vs fully epic}
Let $\mathsf{K}$ be a quasivariety, $\B \in \mathsf{K}$, and $\A \leq \B$ full in $\mathsf{K}$. Then $\A \leq \B$ is not epic in $\mathsf{K}$ iff one of the conditions in the statement of Theorem \ref{Thm : weak ES vs fully epic} holds. Furthermore, if condition (\ref{weak ES vs fully epic : item : 1}) holds, we may assume that $\theta = \textup{id}_B$. 
\end{Proposition}

\begin{proof}
Suppose first that $\A \leq \B$ is not epic in $\mathsf{K}$. Therefore, there exist $\C \in \mathsf{K}$ and two homomorphisms $g, h \colon \B \to \C$ such that $g{\upharpoonright}_A = h{\upharpoonright}_A$ and $g \ne h$. As $g \ne h$ there exists also $b \in B$ such that $g(b) \ne h(b)$. By the Subdirect Decomposition Theorem \ref{Thm : Subdirect Decomposition} there exist $\D \in \mathsf{K}_\textsc{rsi}$ and a homomorphism $f \colon \C \to \D$ such that $f(g(b)) \ne f(h(b))$. Thus, $f\circ g$ and $f \circ h$ differ, but their restrictions to $A$ coincide. Consequently, we may assume that $\C \in \mathsf{K}_\textsc{rsi}$ (otherwise, we replace $\C$ by $\D$ and $h$ and $g$ by their compositions with $f$).

If both $g$ and $h$ are injective, then condition (\ref{weak ES vs fully epic : item : 2}) of Theorem \ref{Thm : weak ES vs fully epic} holds.\ Then we consider the case where one of them is not. By symmetry we may assume that $g$ is not injective, that is, $\mathsf{Ker}(g) \ne \textup{id}_{B}$.\ We will prove that condition (\ref{weak ES vs fully epic : item : 1}) of Theorem \ref{Thm : weak ES vs fully epic} holds.\  Clearly, we have $\mathsf{Ker}(g), \mathsf{Ker}(h) \in \mathsf{Con}_\mathsf{K}(\B)$. Therefore, it suffices to show that
\[
\mathsf{Ker}(g){\upharpoonright}_A = \mathsf{Ker}(h){\upharpoonright}_A \, \, \text{ and } \, \, \mathsf{Ker}(g) \ne \mathsf{Ker}(h).
\]

The first half of the above display holds because $g{\upharpoonright}_A = h{\upharpoonright}_A$.\ To prove the second, recall that $\A \leq \B$ is full in $\mathsf{K}$ and that $\mathsf{Ker}(g) \ne \textup{id}_{B}$. Therefore, there exists $a \in A$ such that $\langle a, b \rangle \in \mathsf{Ker}(g)$. Together with $g{\upharpoonright}_A = h{\upharpoonright}_A$, this yields
\[
g(b) = g(a) = h(a).
\]
Since $g(b) \ne h(b)$, this implies $h(a) \ne h(b)$. Thus, $\langle a, b \rangle \notin \mathsf{Ker}(h)$. Hence, we conclude that $\mathsf{Ker}(g) \ne \mathsf{Ker}(h)$.

Now, we prove the converse. If condition (\ref{weak ES vs fully epic : item : 2}) holds, then it is clear that $\A \leq \B$ is not epic in $\mathsf{K}$. Then suppose that condition (\ref{weak ES vs fully epic : item : 1}) holds. In this case, there are two distinct $\theta, \phi \in \mathsf{Con}_\mathsf{K}(\B)$ such that $\theta{\upharpoonright}_A = \phi{\upharpoonright}_A$. Then 
Lemma~\ref{Lem : technical phi theta separation} yields $\theta \cap \phi = \textup{id}_B$. As $\theta \ne \phi$, we have $\theta \nsubseteq \phi$ or $\phi \nsubseteq \theta$. By symmetry, we may assume that $\phi \nsubseteq \theta$, and hence that $\phi \ne \theta \cap \phi$. 
From $\phi{\upharpoonright}_A = \theta{\upharpoonright}_A$ it follows that $\phi{\upharpoonright}_A = \theta{\upharpoonright}_A \cap \phi{\upharpoonright}_A = (\theta \cap \phi){\upharpoonright}_A$.
Thus, replacing $\theta$ with $\theta \cap \phi$ allows us to assume that $\theta = \textup{id}_B$, $\phi \ne \textup{id}_B$, and  
$\phi{\upharpoonright}_A=(\textup{id}_B){\upharpoonright}_A$.

Since $\A \leq \B$ is full in $\mathsf{K}$, $\phi \ne \textup{id}_B$, and $\phi{\upharpoonright}_A=(\textup{id}_B){\upharpoonright}_A=\textup{id}_A$,  for every $b \in B$ there exists a unique $a_b \in A$ such that $\langle a_b, b \rangle \in \phi$. We will prove that the map $g \colon \B \to \B$ defined by the rule $g(b) \coloneqq a_b$ is a homomorphism. To this end, let $f$ be a basic $n$-ary operation and $b_1, \dots, b_n \in B$.  From $\langle g(b_1), b_1 \rangle, \dots, \langle g(b_n), b_n \rangle \in \phi$ it follows $\langle f^{\B}(g(b_1), \dots, g(b_n)), f^{\B}(b_1, \dots, b_n)\rangle \in \phi$. As $f^{\B}(g(b_1), \dots, g(b_n))$ belongs to $A$, it is the unique element of $A$ which is related to $f^{\B}(b_1, \dots, b_n)$ by $\phi$. Therefore, the definition of $g$ gives
\begin{equation*}
g(f^\B(b_1, \dots, b_n)) = f^{\B}(g(b_1), \dots, g(b_n)).
\end{equation*}
Thus, $g$ is a homomorphism. The definition of $g$ implies $g{\upharpoonright}_A=i{\upharpoonright}_A$, where $i \colon \B \to \B$ is the identity map.
 Moreover, $g \ne i$ because 
$g[\B]=\A$
and $\A \leq \B$ is proper.  Therefore, $\A \leq \B$ is not epic in $\mathsf{K}$.
\end{proof}

Theorem \ref{Thm : weak ES vs fully epic} is now an immediate consequence of Corollary \ref{Cor : fully epic = weak ES} and Proposition \ref{Prop : weak ES vs fully epic}.

\begin{Corollary}\label{Cor : A RFSI}
Let $\mathsf{K}$ be a quasivariety. If $\B \in \mathsf{K}$ and $\A \leq \B$ is full and not epic in $\mathsf{K}$, then $\A \in \SSS(\mathsf{K}_{\textup{\textsc{rsi}}})$.\end{Corollary}

\begin{proof}
As $\A \leq \B$ is not epic in $\mathsf{K}$, the proof of the implication from left to right of Proposition \ref{Prop : weak ES vs fully epic} shows   that there exist $\C \in \mathsf{K}_{\textsc{rsi}}$ and two homomorphisms $g, h \colon \B \to \C$ such that one of the following conditions holds:
\benroman
\item\label{cor : item : injective : 1} $\mathsf{Ker}(g) \ne \mathsf{Ker}(h)$ and $\mathsf{Ker}(g){\upharpoonright}_A = \mathsf{Ker}(h){\upharpoonright}_A$;
\item\label{cor : item : injective : 2} The maps $g$ and $h$ are distinct embeddings.
\eroman

We will prove that in both cases, $g{\upharpoonright}_A$ is injective.\ In case (\ref{cor : item : injective : 2}) this is clear. Then we consider case (\ref{cor : item : injective : 1}). By Lemma \ref{Lem : technical phi theta separation} we have that
$\textup{id}_B = \mathsf{Ker}(g) \cap \mathsf{Ker}(h)$. Together with $\mathsf{Ker}(g){\upharpoonright}_A = \mathsf{Ker}(h){\upharpoonright}_A$, this yields
\[
\textup{id}_A =(\textup{id}_B){\upharpoonright}_A = (\mathsf{Ker}(g) \cap \mathsf{Ker}(h)){\upharpoonright}_A = \mathsf{Ker}(g){\upharpoonright}_A \cap \mathsf{Ker}(h){\upharpoonright}_A = \mathsf{Ker}(g){\upharpoonright}_A.
\]
Therefore, $g{\upharpoonright}_A$ is injective as desired. As a consequence, $g{\upharpoonright}_A \colon \A \to \C$ is an embedding. Since $\C \in \mathsf{K}_{\textsc{rsi}}$, we conclude that  $\A \in \III\SSS(\mathsf{K}_{\textup{\textsc{rsi}}}) = \SSS\III(\mathsf{K}_{\textup{\textsc{rsi}}}) = \SSS(\mathsf{K}_{\textup{\textsc{rsi}}})$ because $\mathsf{K}_{\textup{\textsc{rsi}}}$ is closed under isomorphisms.
\end{proof}

 As we mentioned, every variety of Heyting algebras has the weak ES property. The proof of this fact \cite[Thm.\ 1]{Kreisel60JSL} establishes the logical counterpart of this property (namely, the Beth definability property for intermediate logics).\ A simple algebraic proof can be derived from Theorem \ref{Thm : weak ES vs fully epic} as we proceed to illustrate.

\begin{exa}\label{Exa : Kreisel}
The $\langle \land, \to \rangle$-subreducts of Heyting algebras are called \emph{implicative semilattices} 
\cite{Ho62} (see also \cite{Kh81}).
We will prove that varieties of Heyting algebras and of implicative semilattices have the weak ES property. To this end, we recall that the lattice of filters of a Heyting algebra (resp.\ implicative semilattice) $\A$ is isomorphic to that of its congruences  under the map that takes a filter $F$ to the congruence
\[
\theta_F \coloneqq \{ \langle a, b \rangle \in A \times A : a \to b, b \to a \in F \}
\]
(see, e.g., \cite[Prop.\ 2.4.9(b)]{Esakia-book85} and \cite[p.\ 106]{Kh81}).

Now, let $\mathsf{K}$ be a variety of Heyting algebras or of implicative semilattices. By Theorem \ref{Thm : weak ES vs fully epic}, in order to prove that $\mathsf{K}$ has the weak ES property, it suffices to show that for every $\B \in \mathsf{K}$ and proper $\A \leq \B$ there exist two distinct $\theta, \phi \in \mathsf{Con}(\B)$ such that $\theta{\upharpoonright}_A = \phi{\upharpoonright}_A$.\footnote{In principle, we may also assume that $\A \leq \B$ is full in $\mathsf{K}$ and that $\B$ is finitely generated, but none of these assumptions will be needed to establish the desired conclusion.} Consider $\B \in \mathsf{K}$ and $\A \leq \B$ proper. Let $b \in B - A$ and consider the following filters of $\B$:
\[
F \coloneqq \{ a \in B : b \leq a \} \, \, \text{ and } \, \, G \coloneqq \{ c \in B : \text{there exists }a \in A \text{ s.t. }b \leq a \leq c \}.
\]
As $b \notin A$, we have $b \in F - G$ and, therefore, $F \ne G$. On the other hand, $F \cap A = G \cap A$ by the definition of $F$ and $G$. From $F \ne G$ it follows $\theta_F \ne \theta_G$. Since $\A$ is a subalgebra of $\B$, we have
\begin{align*}
\theta_F{\upharpoonright}_{A} &= \{ \langle a, c \rangle \in A \times A : a \to c, c \to a \in F \cap A \};\\
\theta_G{\upharpoonright}_{A} &= \{ \langle a, c \rangle \in A \times A : a \to c, c \to a \in G \cap A \}.
\end{align*}
Together with $F \cap A = G \cap A$, this yields $\theta_F{\upharpoonright}_{A} = \theta_G{\upharpoonright}_{A}$.
\qed
\end{exa}

\begin{Remark}
Corollary \ref{Cor : A RFSI} cannot be strengthened by concluding that also $\B \in \oper{S}(\mathsf{K}_{\textup{\textsc{rsi}}})$, nor even that $\B \in \oper{S}(\mathsf{K}_{\textup{\textsc{rfsi}}})$. To prove this,
let $\B$ be the semilattice with 
meet-order $0 < a < 1$ and $\A \leq \B$ its subalgebra with universe $\{ 0, 1 \}$.
While $\A \leq \B$ is easily seen to be full in the variety $\mathsf{SL}$ of semilattices, it is not epic because $\mathsf{SL}$ has the weak ES property (see, e.g., \cite[p.\ 99]{KMPT83}). Nonetheless, $\B$ is not a subalgebra of an FSI semilattice because, up to isomorphism, the only FSI semilattice is $\A$.
\qed
\end{Remark}

\section{Quasivarieties with a near unanimity term}

A term $\varphi$ of arity $ \geq 3$ is a \emph{near unanimity term} for a class $\mathsf{K}$ of algebras if
\[
\mathsf{K} \vDash x \thickapprox \varphi(y, x, \dots, x) \thickapprox \varphi(x, y, x, \dots, x) \thickapprox \dots \thickapprox \varphi(x, \dots, x, y).
\]
Ternary near unanimity terms play a prominent role in algebra and are known as \emph{majority terms}. For instance, $(x \land y) \lor (x \land z) \lor (y \land z)$ is a majority term for every class of algebras with a lattice reduct. 

A useful feature of classes with a near unanimity term is the following \cite[Thm.\ 2]{Mitschke78}:

\begin{Theorem}\label{Thm : NUT implies CD}
Let  $\mathsf{K}$ be a class of algebras with a near unanimity term. Then $\VVV(\mathsf{K})$ is congruence distributive.
\end{Theorem}

\noindent However, the converse of Theorem \ref{Thm : NUT implies CD} is not true in general because there are varieties that are congruence distributive and lack a near unanimity term (see, e.g., Example \ref{Exa : implication algebras}). Furthermore,
even when the class $\mathsf{K}$ possesses a near unanimity term, the quasivariety $\QQQ(\mathsf{K})$ need not be congruence distributive.

In the presence of a near unanimity term, the task of determining whether a quasivariety has the ES property is simplified by the following result \cite[Thm.\ 6.4]{Camper18jsl}:\footnote{In the statement of \cite[Thm.\ 6.4]{Camper18jsl}, the occurrences of $n-1$ in Theorem \ref{Thm : Campercholi unanimity} are replaced by $n$. However, its proof yields the stronger version reported here (see \cite{Kurtzhals2024}).}

\begin{Theorem}\label{Thm : Campercholi unanimity}
Let $\mathsf{K}$ be a quasivariety with a near unanimity term of arity $n$. Then $\mathsf{K}$ has the ES property iff every $\A \leq \A_1 \times \dots \times \A_{n-1}$ with $\A_1, \dots, \A_{n-1} \in \PPU(\mathsf{K}_{\textup{\textsc{rsi}}})$ lacks subalgebras that are proper and epic in $\mathsf{K}$.
\end{Theorem}

The aim of this section is to show that this result can be improved significantly for the case of the weak ES property. More precisely, we will prove the following:

\begin{Theorem}\label{Thm : near unanimity}
Let $\mathsf{K}$ be a quasivariety with a near unanimity term of arity $n$. Then $\mathsf{K}$ has the weak ES property iff every finitely generated subdirect product $\A \leq \A_1 \times \dots \times \A_{n-1}$ with $\A_1, \dots, \A_{n-1} \in \mathsf{K}_{\textup{\textsc{rfsi}}}$ lacks subalgebras that are fully epic in $\mathsf{K}$.
\end{Theorem}

To prove this theorem, it is convenient to introduce the following concept:

\begin{Definition}
Let $\mathsf{K}$ be a quasivariety, $\A \in \mathsf{K}$, and $\theta \in \mathsf{Con}_\mathsf{K}(\A)$.\ Given a positive integer $n$, we say that $\theta$ is \emph{$n$-irreducible} in $\mathsf{Con}_\mathsf{K}(\A)$ when $\theta= \theta_1 \cap \dots \cap \theta_n$ with $\theta_1, \dots, \theta_n \in \mathsf{Con}_\mathsf{K}(\A)$ implies $\theta=\theta_1 \cap \dots \cap \theta_{i-1} \cap \theta_{i+1} \cap \dots \cap \theta_n$ for some $i \leq n$. When $\mathsf{K}$ is clear from the context, we will simply say that $\theta$ is \emph{$n$-irreducible}.
\end{Definition}

Notice that the only $1$-irreducible $\mathsf{K}$-congruence of $\A$ is $A \times A$. Moreover, a $\mathsf{K}$-congruence $\theta$ of $\A$ is $2$-irreducible if and only if either $\theta \in \mathsf{Irr}_{\mathsf{K}}(\A)$ or $\theta = A \times A$.

\begin{Proposition}\label{Prop : optimal two cases}
Let $\mathsf{K}$ be a quasivariety, $\A \in \mathsf{K}$, and $\theta \in \mathsf{Con}_\mathsf{K}(\A)$ $n$-irreducible. Then there exist $\phi_1, \dots, \phi_{n-1} \in \mathsf{Irr}_{\mathsf{K}}(\A)$ such that $\theta = \phi_1 \cap \dots \cap \phi_{n-1}$.
\end{Proposition}

\begin{proof}
Let $m$ be the least positive integer such that $\theta$ is $m$-irreducible. Since we allow repetitions among the $\phi_1, \dots, \phi_{n-1}$ in the statement and $m \leq n$, it is sufficient to show that there exist $\phi_1, \dots, \phi_{m-1} \in \mathsf{Irr}_{\mathsf{K}}(\A)$ such that $\theta = \phi_1 \cap \dots \cap \phi_{m-1}$.
If $m=1$, then $\theta$ is $1$-irreducible. Thus, $\theta=A \times A$, and $\theta$ can be written as the intersection of an empty family of members of $\mathsf{Irr}_{\mathsf{K}}(\A)$. 
So, we may assume that $m \geq 2$. 

As $\theta$ is not $(m-1)$-irreducible, there exist $\theta_1, \dots, \theta_{m-1} \in \mathsf{Con}_\mathsf{K}(\A)$ such that $\theta = \theta_1 \cap \dots \cap \theta_{m-1}$ and $\theta \ne \theta_1 \cap \dots \cap \theta_{i-1} \cap \theta_{i+1} \cap \dots \cap \theta_{m-1}$ for every $i \leq m-1$. Consider the poset
\[
X \coloneqq \{ \langle \phi_1, \dots, \phi_{m-1} \rangle  : \theta_i \subseteq \phi_i \in \mathsf{Con}_\mathsf{K}(\A) \text{ for every $i \leq m-1$ and } \theta = \phi_1 \cap \dots \cap \phi_{m-1} \}
\]
ordered under the relation given by $\langle \phi_1, \dots, \phi_{m-1} \rangle \leq \langle \eta_1, \dots, \eta_{m-1} \rangle$ if and only if $\phi_i \subseteq \eta_i$ for every $i \leq m-1$. 
We will apply Zorn's Lemma to obtain a maximal element of $X$. Clearly, $X$ contains $\langle\theta_1, \dots, \theta_{m-1}\rangle$. 
Consider a nonempty chain $C$ in $X$. For each $i \leq m-1$ let $C_i$ be the projection of $C$ on the $i$-th coordinate. Observe that $C_i$ is a nonempty chain in $\mathsf{Con}_\mathsf{K}(\A)$ because $C$ is a nonempty chain in $X$. We will prove that $\langle \bigcup C_1, \dots, \bigcup C_{m-1}\rangle$ is an upper bound of $C$ in $X$. Proposition \ref{Prop : chains of congruences} implies that each $\bigcup C_i$ is a $\mathsf{K}$-congruence of $\A$. Furthermore, as $\theta_i$ is contained in every member of $C_i$ and $C_i$ is nonempty, we have $\theta_i \subseteq \bigcup C_i$ for every $i \leq m-1$.
Lastly, observe that
\begin{equation}\label{eq Ci}
\left(\bigcup C_1\right) \cap \dots \cap \left(\bigcup C_{m-1}\right) = \bigcup \{ \phi_1 \cap \dots \cap \phi_{m-1} : \phi_i \in C_i \text{ for } i \leq m-1\} = \theta_1 \cap \dots \cap \theta_{m-1},
\end{equation}
where the first equality holds by the infinite distributive law and the second can be established as follows. Let $\phi_1 \in C_1, \dots, \phi_{m-1} \in C_{m-1}$. Then there exists $\langle \phi_1', \dots, \phi_{m-1}' \rangle \in C$ such that $\phi_i \subseteq \phi_i'$ for every $i \leq m-1$ because $C$ is a chain. As a consequence,
\[
\theta_1 \cap \dots \cap \theta_{m-1} \subseteq \phi_1 \cap \dots \cap \phi_{m-1} \subseteq \phi_1' \cap \dots \cap \phi_{m-1}' 
=\theta,
\]
where the the first inclusion holds because $\phi_i \in C_i$ for every $i \leq m-1$, and the last equality holds because $\langle \phi_1', \dots, \phi_{m-1}' \rangle \in C$. Hence, $\phi_1 \cap \dots \cap \phi_{m-1} =  \theta$
 because $\theta = \theta_1 \cap \dots \cap \theta_{m-1}$. 
Therefore, \eqref{eq Ci} holds, and so $\langle \bigcup C_1, \dots, \bigcup C_{m-1}\rangle$ is an upper bound of $C$ in $X$. 
By Zorn's Lemma the poset $X$ has a maximal element $\langle \phi_1, \dots, \phi_{m-1} \rangle$. In particular, $\theta = \phi_1 \cap \dots \cap \phi_{m-1}$. 

It only remains to show that $\phi_1, \dots, \phi_{m-1} \in \mathsf{Irr}_{\mathsf{K}}(\A)$. First observe that for every $i \leq m-1$ we have
\[
\theta \ne  \phi_1 \cap \dots \cap \phi_{i-1} \cap \phi_{i+1} \cap \dots \cap \phi_{m-1},
\]
because otherwise $\theta =  \theta_1 \cap \dots \cap \theta_{i-1} \cap \theta_{i+1} \cap \dots \cap \theta_{m-1}$ as $\theta \subseteq \theta_j \subseteq \phi_j$ for every $j \leq m-1$.
It follows that $\phi_i \ne A \times A$ for every $i \leq m-1$ because $\theta =  \phi_1 \cap \dots \cap \phi_{m-1}$. 
Now, suppose that $\phi_i=\eta_1 \cap \eta_2$ for some $\eta_1,\eta_2 \in \mathsf{Con}_\mathsf{K}(\A)$. We have 
\[
\theta = \phi_1 \cap \dots \cap \phi_{m-1} = \phi_1 \cap \dots \cap \phi_{i-1} \cap \eta_1 \cap \eta_2 \cap \phi_{i+1} \cap \dots \cap \phi_{m-1}.
\]
As $\theta$ is $m$-irreducible and $\theta \ne \phi_1 \cap \dots \cap \phi_{k-1} \cap \phi_{k+1} \cap \dots \cap \phi_{m-1}$ for every $k \leq m-1$, we obtain 
\[
\theta = \phi_1 \cap \dots \cap \phi_{i-1} \cap \eta_j \cap \phi_{i+1} \cap \dots \cap \phi_{m-1}
\]
for $j=1$ or $j=2$.
Since $\theta_i \subseteq \phi_i \subseteq \eta_1, \eta_2$, the maximality of $\langle \phi_1, \dots, \phi_{m-1} \rangle$ in $X$ implies that $\phi_i=\eta_1$ or $\phi_i=\eta_2$. Thus, $\phi_i\in \mathsf{Irr}_{\mathsf{K}}(\A)$ as desired. 
\end{proof}

We will also make use of the following observation:

\begin{Proposition}\label{Prop : near unanimity full}
Let $\mathsf{K}$ be a quasivariety with a near unanimity term of arity $n$. Moreover, let $\B \in \mathsf{K}$ and $\A \leq \B$ full in $\mathsf{K}$. Then there exists a subdirect embedding of $\B$ into $\B_1 \times \dots \times \B_{n-1}$ for some $\B_1, \dots, \B_{n-1} \in \mathsf{K}_\textup{\textsc{rfsi}}$.
\end{Proposition}

\begin{proof}
We first show that $\text{id}_B$ is $n$-irreducible. Let $\theta_1, \dots, \theta_n \in \mathsf{Con}_{\mathsf{K}}(\B)$ be such that $\text{id}_B = \theta_1 \cap \dots \cap \theta_n$. Let also $\phi_i = \theta_1 \cap \dots \cap \theta_{i-1} \cap \theta_{i+1} \dots \cap \theta_n$ for each $i \leq n$. We will show that $\phi_i=\text{id}_B$ for some $i \leq n$. Suppose the contrary, with a view to contradiction. Now, recall that $\A\leq \B$ is proper and almost total. Therefore, there exists $b \in B$ such that $b \notin A$ and $B = \textup{Sg}^\B(A \cup \{ b \})$. Since $\A \leq \B$ is full in $\mathsf{K}$, there exist $a_1, \dots, a_n \in A$ such that $\langle a_i, b \rangle \in \phi_i$ for every $i \leq n$. By assumption $\mathsf{K}$ has a near unanimity term $\varphi(x_1, \dots, x_n)$. We will prove that
\[
\langle \varphi^\B(a_1, \dots, a_n), b \rangle \in  \theta_j \quad \text{for every } j \leq n.
\]
To this end, consider $j \leq n$. As $\langle a_i, b \rangle \in \phi_i \subseteq \theta_j$ for every $i \leq n$ such that $i \ne j$, we obtain $\langle \varphi^\B(a_1, \dots, a_n), \varphi^\B(b, \dots, b, a_j, b, \dots, b) \rangle \in  \theta_j$.  Furthermore, since $\varphi$ is a near unanimity term, we have  $\varphi^\B(b, \dots, b, a_j, b, \dots, b)=b$.  Hence, $\langle \varphi^\B(a_1, \dots, a_n), b \rangle \in  \theta_j$. This establishes the above display. 
Together with the assumption that $\text{id}_B = \theta_1 \cap \dots \cap \theta_n$, this implies $b = \varphi^\B(a_1, \dots, a_n)$. As $a_1, \dots, a_n \in A$ and $\A \leq \B$, we conclude that $b \in A$, which is false. Hence, $\text{id}_B$ is $n$-irreducible as desired.

By Proposition \ref{Prop : optimal two cases} there exist $\theta_1, \dots, \theta_{n-1} \in \mathsf{Irr}_{\mathsf{K}}(\B)$ such that $\text{id}_B =  \theta_1 \cap \dots \cap \theta_{n-1}$. Therefore, we can apply Proposition \ref{Prop : subdirect embedding} obtaining a subdirect embedding
\[
f \colon \B \to \B / \theta_1 \times \dots \times \B / \theta_{n-1}.
\]
Furthermore, from Proposition \ref{Prop : RFSI} and $\theta_1, \dots, \theta_{n-1} \in \mathsf{Irr}_{\mathsf{K}}(\B)$ it follows that each $\B / \theta_i$ belongs to $\mathsf{K}_{\textsc{rfsi}}$. 
\end{proof}

We are now ready to prove Theorem \ref{Thm : near unanimity}.

\begin{proof}[Proof of Theorem \ref{Thm : near unanimity}.]
The implication from left to right holds by Proposition \ref{Prop : epic}.  
We will prove the other implication by contraposition.
Suppose that $\mathsf{K}$ lacks the weak ES property. By Corollary \ref{Cor : fully epic = weak ES} there exist $\B \in \mathsf{K}$ finitely generated and $\A \leq \B$ fully epic in $\mathsf{K}$. In view of Proposition \ref{Prop : near unanimity full}, we may assume that $\B \leq \B_1 \times \dots \times \B_{n-1}$ is a subdirect product for some $\B_1, \dots, \B_{n-1} \in \mathsf{K}_{\textsc{rfsi}}$.
\end{proof}

As we mentioned, Theorem \ref{Thm : near unanimity} simplifies the task of determining whether a quasivariety has the weak ES property. The next example illustrates this in the setting of lattice theory.

\begin{exa}\label{Exa : distributive lattices}
Given $n \in \oper{Z}^+$, we say that a lattice $\A$ has \emph{length} $\leq n$ when its chains have cardinality $\leq n$. Similarly, we say that $\A$ has \emph{length} $n$ when it has length $\leq n$ and it contains an $n$-element chain. Lastly, a class $\mathsf{K}$ of lattices has \emph{bounded length} when there exists some $n \in \oper{Z}^+$ such that every member of $\mathsf{K}$ has length $\leq n$. We will prove that every nontrivial variety of lattices generated by a class of bounded length  lacks the weak ES property. As a consequence, every finitely generated nontrivial variety of lattices also lacks the weak ES property. Since varieties of lattices have a majority term, by Theorem \ref{Thm : near unanimity} we expect these failures of the weak ES property to be witnessed by fully epic situations $\A \leq \B$ where $\B$ can be viewed as a subdirect product $\B \leq \A_1 \times \A_2$ with $\A_1$ and $\A_2$ FSI.

Let $\mathsf{K}$ be a class of lattices of bounded length such that $\VVV(\mathsf{K})$ is nontrivial. By
Theorem \ref{Jonsson Lemma},
the FSI members of $\VVV(\mathsf{K})$ belong 
to $\HHH\SSS\PPU(\mathsf{K})$. Since 
the members of $\mathsf{K}$ have length $\leq n$ for 
some $n \in \oper{Z}^+$, so do those of $\PPU(\mathsf{K})$ by 	\L o\'s' Theorem \cite[Thm.\ V.2.9]{BuSa00}. As the class operators $\HHH$ and $\SSS$ do not increase the cardinality of chains, we conclude that the FSI members of $\VVV(\mathsf{K})$ have length $\leq n$. Now, as $\VVV(\mathsf{K})$ is nontrivial, it has at least one SI member. Therefore, there exists  $\A_1 \in \VVV(\mathsf{K})_{\textsc{si}}$ of maximal length. Let $\textup{Cg}^{\A_1}(a, b)$
be the monolith of $\A_1$.  As $a \ne b$, we may assume that $a < b$. Moreover, let $\A_2$ be the two-element chain $0 < 1$ viewed as a lattice,
which is also SI. Consider the subalgebra $\A$ of $\A_1 \times \A_2$ with universe $(A_1 \times \{ 0 \}) \cup ({\uparrow}b \times \{ 1 \})$, where ${\uparrow}b \coloneqq \{ c \in A_1 : b \leq c \}$, and let $\B \coloneqq \textup{Sg}^{\A_1 \times \A_2}(A \cup \{ \langle a, 1 \rangle\})$.\ Clearly, $\B \leq \A_1 \times \A_2$ is a subdirect product.\ We will prove that $\A \leq \B$ is fully epic in $\VVV(\mathsf{K})$, thus showing that $\VVV(\mathsf{K})$ lacks the weak ES property.

First observe that for every nonidentity $\theta \in \mathsf{Con}(\B)$ we have
\[
(\langle a, 1 \rangle/ \theta = \langle b, 1 \rangle / \theta \text{ and } \langle a, 0 \rangle / \theta = \langle b, 0 \rangle / \theta)\text{ or }(\langle a, 1 \rangle / \theta = \langle a, 0 \rangle / \theta \text{ and } \langle b, 1 \rangle / \theta = \langle b, 0 \rangle / \theta).
\]
To prove this, consider a nonidentity $\theta \in \mathsf{Con}(\B)$. As $\VVV(\mathsf{K})$ is congruence distributive by Theorem \ref{Thm : NUT implies CD} and $\B \leq \A_1 \times \A_2$ is a subdirect product, we can apply Theorem \ref{Thm : CD varieties} obtaining some $\theta_1 \in \mathsf{Con}(\A_1)$ and $\theta_2 \in \mathsf{Con}(\A_2)$ such that $\theta = (\theta_1 \times \theta_2){\upharpoonright}_B$. As $\theta \ne \textup{id}_B$, some $\theta_i$ contains the monolith of $\A_i$. Thus, $\langle a, b \rangle \in \theta_1$ or $\langle 0, 1 \rangle \in \theta_2$. Together with $\theta = (\theta_1 \times \theta_2){\upharpoonright}_B$, this establishes the above display.

Now, as $\langle a, 1 \rangle \notin A$ and $\B$ is generated by $A \cup \{ \langle a, 1 \rangle\}$, in order to prove that $\A \leq \B$ is full in $\VVV(\mathsf{K})$, it suffices to show that every nonidentity congruence of $\B$ glues $\langle a, 1 \rangle$ with some element of $\A$. But this holds by the above display. As $\A \leq \B$ is full in $\VVV(\mathsf{K})$, we can apply Proposition \ref{Prop : weak ES vs fully epic} obtaining that $\A \leq \B$ is  epic in $\VVV(\mathsf{K})$ provided that the two conditions in Theorem \ref{Thm : weak ES vs fully epic} fail.\ If condition (\ref{weak ES vs fully epic : item : 1}) holds, the last part of Proposition \ref{Prop : weak ES vs fully epic} implies that there exists a nonidentity congruence $\theta$ of $\B$ with $\theta{\upharpoonright}_A = \textup{id}_A$, a contradiction with the above display. On the other hand, if condition (\ref{weak ES vs fully epic : item : 2}) holds, $\B$ embeds into an SI member of $\VVV(\mathsf{K})$. But this is impossible because the length of $\B$ is strictly larger than the length of $\A_1$, and the latter is an upper bound for the length of the SI members of $\VVV(\mathsf{K})$.
\qed
\end{exa}

We close this section by providing evidence suggesting that Theorem \ref{Thm : near unanimity} cannot be extended beyond the setting of quasivarieties with a near unanimity term. More precisely, we will show that Proposition \ref{Prop : near unanimity full} fails for arbitrary congruence distributive varieties.

\begin{exa}\label{Exa : implication algebras}
The implicative subreducts
of Boolean algebras are called \emph{implication algebras} \cite{Ab67}. These form a variety $\mathsf{IA}$ that is congruence distributive, but lacks a near unanimity term \cite[Lem.\ 3]{Mitschke78}. We will prove that Proposition \ref{Prop : near unanimity full} fails for $\mathsf{IA}$.\ Let $\B$ be the implicative reduct 
of the powerset Boolean algebra $\mathcal{P}(\oper{N})$ and $\A$ the subalgebra of $\B$ with universe $\mathcal{P}(\oper{N}) - \{ \emptyset \}$. We will show that $\A \leq \B$ is full in $\mathsf{IA}$, but that $\B$ cannot be obtained as a subdirect product $\B \leq \A_1 \times \dots \times \A_n$ in which each $\A_i$ is FSI. The latter holds because the only FSI implication algebra is the implicative reduct 
of the two-element Boolean algebra, while $\B$ is infinite. Then we prove that $\A \leq \B$ is full in $\mathsf{IA}$. Clearly, $\A \leq \B$ is proper and almost total. Moreover, as the congruences of $\B$ coincide with those of the powerset Boolean algebra $\mathcal{P}(\oper{N})$, every nonidentity congruence of $\B$ glues $\emptyset$ with some element of $\A$.\footnote{Although we do not need it here, we remark that $\A \leq \B$ is indeed fully epic in $\mathsf{IA}$ (cf.\ \cite[Prop.\ 4.5]{BlHoo06}).}
\qed
\end{exa}

\section{Congruence permutable varieties}

Given two binary relations $R_1$ and $R_2$ on a set $A$, we let
\[
R_1 \circ R_2 \coloneqq \{ \langle a, b \rangle \in A \times A : \text{there exists }c\in A \text{ s.t. }\langle a, c \rangle \in R_1 \text{ and }\langle c, b \rangle \in R_2 \}.
\]
A variety $\mathsf{K}$ is said to be \emph{congruence permutable} when for every $\A \in \mathsf{K}$ and $\theta_1, \theta_2 \in \mathsf{Con}(\A)$ we have $\theta_1 \circ \theta_2 = \theta_2 \circ \theta_1$. As the notion of congruence permutability does not generalize smoothly to quasivarieties, in this section we will focus on varieties only. Given an algebra $\A$, we denote the join operation of the lattice $\mathsf{Con}(\A)$  by $+^\A$. We will make use of the following observation (see, e.g., \cite[Thm.\ II.5.9]{BuSa00}):

\begin{Proposition}\label{Prop : permutable joins relational}
A variety $\mathsf{K}$ is congruence permutable iff $\theta_1 +^\A \theta_2 = \theta_1 \circ \theta_2$ for every $\A \in \mathsf{K}$ and $\theta_1, \theta_2 \in \mathsf{Con}(\A)$.
\end{Proposition}

Varieties that are both congruence distributive and congruence permutable are called \emph{arithmetical}. Moreover, a class of algebras is said to be \emph{universal} when it is closed under $\III, \SSS$, and $\PPU$ or, equivalently, when it can be axiomatized by universal sentences (see, e.g., \cite[Thm.\ V.2.20]{BuSa00}). The task of determining whether an arithmetical variety, whose class of FSI members is universal, has the ES property can be simplified as follows \cite[Thm.\ 6.8]{Camper18jsl}:

\begin{Theorem}\label{Thm : Campercholi CP}
Let $\mathsf{K}$ be an arithmetical variety such that $\mathsf{K}_{\textup{\textsc{fsi}}}$ is a universal class. Then $\mathsf{K}$ has the ES property iff the members of $\mathsf{K}_{\textup{\textsc{fsi}}}$ lack subalgebras that are proper and epic in $\mathsf{K}$.
\end{Theorem}

For the case of the weak ES property, this result can be improved as follows:

\begin{Theorem}\label{Thm : CP locally finite}
Let $\mathsf{K}$ be a congruence permutable variety. Then $\mathsf{K}$ has the weak ES property iff the finitely generated members of $\mathsf{K}_{\textup{\textsc{fsi}}}$ lack subalgebras that are fully epic in $\mathsf{K}$.
\end{Theorem}

\begin{proof}
The implication from left to right follows from Proposition \ref{Prop : epic}. To prove the other implication, we reason by contraposition. Suppose that $\mathsf{K}$ lacks the weak ES property. By Corollary \ref{Cor : fully epic = weak ES} there exist a finitely generated $\B \in \mathsf{K}$ and $\A \leq \B$ fully epic in $\mathsf{K}$. Using the next proposition (which will be established in the rest of the section) we conclude that $\B \in \mathsf{K}_{\textsc{fsi}}$. 
\end{proof}

\begin{Proposition} \label{Prop : cong permutable x}
    Let $\mathsf{K}$ be a congruence permutable variety and $\B \in \mathsf{K}$. If  $\A \leq \B$ is fully epic in $\mathsf{K}$, then $\B$ is FSI.
\end{Proposition}

The rest of the section is devoted to the proof of Proposition \ref{Prop : cong permutable x}. We begin with the next observations:

\begin{Lemma}\label{Lem : congr are extensions if fully epic}
Let $\mathsf{K}$ be a variety, $\B \in \mathsf{K}$, and $\A \leq \B$ fully epic in $\mathsf{K}$. Then $\theta=\textup{Cg}^\B(\theta{\upharpoonright}_{A})$ for every $\theta \in \mathsf{Con}(\B)$.
\end{Lemma}

\begin{proof}
Let $\theta \in \mathsf{Con}(\B)$. We first show that $\theta{\upharpoonright}_{A}=\textup{Cg}^\B(\theta{\upharpoonright}_{A}){\upharpoonright}_{A}$. The inclusion $\theta{\upharpoonright}_{A} \subseteq \textup{Cg}^\B(\theta{\upharpoonright}_{A}){\upharpoonright}_{A}$ is straightforward. The other inclusion follows from $\textup{Cg}^\B(\theta{\upharpoonright}_{A}) \subseteq \theta$, which holds because $\theta$ is a congruence of $\B$ containing $\theta {\upharpoonright}_{A}$ and $\textup{Cg}^\B(\theta{\upharpoonright}_{A})$ is the least such. This establishes that $\theta{\upharpoonright}_{A}=\textup{Cg}^\B(\theta{\upharpoonright}_{A}){\upharpoonright}_{A}$.
Since $\A \leq \B$ is fully epic in $\mathsf{K}$, Proposition \ref{Prop : weak ES vs fully epic} implies that $\theta=\textup{Cg}^\B(\theta{\upharpoonright}_{A})$. 
Indeed, if $\theta \ne \textup{Cg}^\B(\theta {\upharpoonright}_A)$, then condition (\ref{weak ES vs fully epic : item : 1}) of Theorem \ref{Thm : weak ES vs fully epic} would hold, contradicting that $\A \leq \B$ is epic.
\end{proof}

\begin{Lemma}\label{Lem : congr can be extended if full}
Let $\mathsf{K}$ be a variety, $\B \in \mathsf{K}$, and $\A \leq \B$ full in $\mathsf{K}$. If $\phi \in \mathsf{Con}(\A)$ and there exists $\theta \in \mathsf{Con}(\B)$ such that $\theta \ne\textup{id}_B$ and $\theta{\upharpoonright}_{A} \subseteq \phi$, then $\phi= \textup{Cg}^\B(\phi){\upharpoonright}_{A}$.
\end{Lemma}

\begin{proof}
Assume that $\phi \in \mathsf{Con}(\A)$ and $\theta \in \mathsf{Con}(\B)$ are such that $\theta \ne\textup{id}_B$ and $\theta{\upharpoonright}_{A} \subseteq \phi$. We first construct $\eta \in \mathsf{Con}(\B)$ such that $\phi= \eta{\upharpoonright}_{A}$. Let 
\[
\eta \coloneqq \{ \langle b_1,b_2 \rangle \in B \times B : \text{there exists }\langle a_1,a_2 \rangle \in \phi\text{ such that }\langle a_1,b_1 \rangle, \langle a_2,b_2 \rangle \in \theta\}.
\]
First, we show that $\eta$ is an equivalence relation. Since $\theta \ne\textup{id}_B$ and $\A \leq \B$ is full in $\mathsf{K}$, it follows that $\eta$ is reflexive. The symmetry of $\eta$ is a consequence of the symmetry of $\phi$. To prove the transitivity of $\eta$, suppose that $\langle b_1,b_2\rangle, \langle b_2,b_3 \rangle \in \eta$. Then there exist $a_1,a_2,a_2',a_3 \in A$ such that $\langle a_1,a_2\rangle, \langle a_2',a_3 \rangle \in \phi$ and $\langle a_i,b_i \rangle, \langle a_2',b_2\rangle \in \theta$ for $i = 1,2, 3$. Thus, $\langle a_2, a_2' \rangle \in \theta{\upharpoonright}_{A}$, and hence $\langle a_2, a_2' \rangle \in \phi$ because $\theta{\upharpoonright}_{A} \subseteq \phi$. 
Then the transitivity of $\phi$ implies $\langle a_1,a_3\rangle \in \phi$, and the definition of $\eta$ yields that $\langle b_1,b_3\rangle \in \eta$. Therefore, $\eta$ is an equivalence relation.
That $\eta$ is a congruence of $\B$ is then a straightforward consequence of the fact that $\phi$ and $\theta$ are congruences. 
We now show that $\phi= \eta{\upharpoonright}_{A}$. The inclusion from left to right is immediate.
For the other inclusion, assume that $\langle a_1, a_2 \rangle \in \eta$ with $a_1,a_2 \in A$. Then there exists $\langle a_1',a_2' \rangle \in \phi$ such that $\langle a_1',a_1 \rangle, \langle a_2',a_2 \rangle \in \theta$. Thus, $\langle a_1, a_2 \rangle \in \phi$ because $\theta{\upharpoonright}_{A} \subseteq \phi$. Therefore, $\eta \in \mathsf{Con}(\B)$ and $\phi= \eta{\upharpoonright}_{A}$ as desired. 

It only remains to show that $\phi= \textup{Cg}^\B(\phi){\upharpoonright}_{A}$. The inclusion $\phi \subseteq \textup{Cg}^\B(\phi){\upharpoonright}_{A}$ is clear. The other inclusion follows from $\textup{Cg}^\B(\phi){\upharpoonright}_{A} \subseteq \eta{\upharpoonright}_{A} = \phi$. 
\end{proof}

Lastly, we will require the following observation:

\begin{Proposition}\label{Prop : Cg commutes with restrictions}
Let $\mathsf{K}$ be a variety, $\B \in \mathsf{K}$, and $\A \leq \B$ fully epic in $\mathsf{K}$.
For every $\theta_1, \theta_2 \in \mathsf{Con}(\B)$ we have
   \[
(\theta_1 +^\B \theta_2){\upharpoonright}_{A} = \theta_1{\upharpoonright}_{A} +^\A  \theta_2{\upharpoonright}_{A}.
   \] 
\end{Proposition}

\begin{proof} 
The equality trivially holds when $\theta_1=\textup{id}_B$, so we may assume that $\theta_1 \ne \textup{id}_B$.
Since $\A \leq \B$ is full in $\mathsf{K}$ and $\theta_1{\upharpoonright}_{A} \subseteq \theta_1{\upharpoonright}_{A} +^\A  \theta_2{\upharpoonright}_{A}$, Lemma \ref{Lem : congr can be extended if full} implies
\[
\theta_1{\upharpoonright}_{A} +^\A  \theta_2{\upharpoonright}_{A} = \textup{Cg}^\B(\theta_1{\upharpoonright}_{A} +^\A  \theta_2{\upharpoonright}_{A}){\upharpoonright}_{A}.
\]
Consequently, to prove the statement, it suffices to show that
\[
\textup{Cg}^\B(\theta_1{\upharpoonright}_{A} +^\A  \theta_2{\upharpoonright}_{A}) = \theta_1 +^\B \theta_2.
\] 
We have that
\[
\textup{Cg}^\B(\theta_1{\upharpoonright}_{A} +^\A  \theta_2{\upharpoonright}_{A})=\textup{Cg}^\B(\theta_1{\upharpoonright}_{A}) +^\B  \textup{Cg}^\B(\theta_2{\upharpoonright}_{A}) = \theta_1 +^\B \theta_2,
\]
where the first equality follows from the fact that $\textup{Cg}^\B(\theta_1{\upharpoonright}_{A} +^\A  \theta_2{\upharpoonright}_{A})$ 
and $\textup{Cg}^\B(\theta_1{\upharpoonright}_{A}) +^\B  \textup{Cg}^\B(\theta_2{\upharpoonright}_{A})$
are both the smallest congruence of $\B$ containing $\theta_1{\upharpoonright}_{A}$ and $\theta_2{\upharpoonright}_{A}$, 
and the second equality is a consequence of Lemma \ref{Lem : congr are extensions if fully epic}. 
\end{proof}

We are now ready to prove Proposition  \ref{Prop : cong permutable x}.

\begin{proof}[Proof of Proposition  \ref{Prop : cong permutable x}.]
Suppose that  $\A \leq \B$ is fully epic in $\mathsf{K}$.\ We need to show that $\B$ is FSI. By Proposition \ref{Prop : RFSI} it suffices to prove that $\textup{id}_B \in \mathsf{Irr}_{\mathsf{K}}(\B)$. First observe that $\textup{id}_B \ne B \times B$ because $\A \leq \B$ is proper. Then let $\theta_1, \theta_2 \in \mathsf{Con}(\B)$ be such that $\textup{id}_B=\theta_1 \cap \theta_2$. Suppose, with a view to contradiction, that $\theta_1, \theta_2 \ne\textup{id}_B$. 
As $\A \leq \B$ is full in $\mathsf{K}$, there exist $b \in B - A$ and $a_1,a_2 \in A$ such that $\langle a_1,b\rangle \in \theta_1$ and $\langle a_2,b\rangle \in \theta_2$. Thus, $\langle a_1,a_2 \rangle \in (\theta_1 +^\B \theta_2){\upharpoonright}_{A}$.
Proposition \ref{Prop : Cg commutes with restrictions} implies that $\langle a_1,a_2 \rangle \in \theta_1{\upharpoonright}_{A} +^\A  \theta_2{\upharpoonright}_{A}$. Since $\mathsf{K}$ is congruence permutable and $\A \in \mathsf{K}$, we can apply Proposition \ref{Prop : permutable joins relational} obtaining
\[
\langle a_1,a_2 \rangle \in \theta_1{\upharpoonright}_{A} +^\A  \theta_2{\upharpoonright}_{A} = \theta_1{\upharpoonright}_{A} \circ \theta_2{\upharpoonright}_{A}.
\]
Therefore, there exists $a \in A$ such that $\langle a_1,a \rangle \in \theta_1$ and $\langle a,a_2 \rangle \in \theta_2$. From $\langle a_1,b\rangle \in \theta_1$ and $\langle a_2,b\rangle \in \theta_2$ it follows that $\langle a,b \rangle \in \theta_1 \cap \theta_2$. Since $\theta_1 \cap \theta_2=\textup{id}_B$, we obtain that $b=a \in A$, a contradiction with $b \notin A$. Thus, $\theta_i =\textup{id}_B$ for some $i=1,2$. Hence, we conclude that $\textup{id}_B \in \mathsf{Irr}_{\mathsf{K}}(\B)$.
\end{proof}

\begin{Remark}
Theorem \ref{Thm : CP locally finite} cannot be strengthened by dropping the assumption that the variety $\mathsf{K}$ is congruence permutable.\ For the variety of distributive lattices $\mathsf{DL}$ lacks the weak ES property 
 because it is finitely generated  (Example~\ref{Exa : distributive lattices}).  On the other hand,  its only FSI member (i.e., the two-element chain) lacks proper subalgebras that are epic in $\mathsf{DL}$. 

Similarly, Proposition \ref{Prop : cong permutable x} cannot be strengthened by assuming that $\A \leq \B$ is only full in $\mathsf{K}$ (as opposed to fully \emph{epic} in $\mathsf{K}$). For let $\mathsf{BA}$ be the variety of Boolean algebras (which is congruence permutable). Moreover, let $\B$ be the four-element Boolean algebra and $\A$ its two-element subalgebra. Then $\A \leq \B$ is full in $\mathsf{BA}$, but $\B$ is not FSI.
\qed
\end{Remark}

\section{Arithmeticity is often necessary}

The aim of this section is to establish the following:

\begin{Theorem}\label{Thm : weak ES implies CP}
Let $\mathsf{K}$ be a congruence distributive quasivariety for which $\mathsf{K}_{\textup{\textsc{rfsi}}}$ is closed under nontrivial subalgebras. If $\mathsf{K}$ has the weak ES property, then $\VVV(\mathsf{K})$ is arithmetical.
\end{Theorem}

We will use the following syntactical description of arithmetical varieties (see, e.g., \cite[Thm.\ II.12.5]{BuSa00}):

\begin{Theorem}\label{Thm : Pixley arithmetical}
A variety $\mathsf{K}$ is arithmetical iff it has a \emph{Pixley term}, that is, a term $\varphi(x, y, z)$ such that
\[
\mathsf{K} \vDash \varphi(x, y, x) \thickapprox \varphi(x, y, y) \thickapprox \varphi(y, y, x) \thickapprox x.
\]
\end{Theorem}

We recall that quasivarieties contain free algebras. More precisely, for every quasivariety $\mathsf{K}$ and nonempty set $X$ there exists an algebra $\boldsymbol{T}_\mathsf{K}(X)$ that is free in $\mathsf{K}$ over $X$ (see, e.g., \cite[Prop.\ 2.1.10]{Go98a}). When $X = \{ x, y \}$, we will write $\boldsymbol{T}_\mathsf{K}(x, y)$ instead of $\boldsymbol{T}_\mathsf{K}(X)$. 

\begin{Proposition}\label{Prop : congruence permutable}
Let $\mathsf{K}$ be a quasivariety and $\boldsymbol{T} \coloneqq \boldsymbol{T}_\mathsf{K}(x, y)$.   Then $\VVV(\mathsf{K})$ is arithmetical iff
\[
\langle x, x, x \rangle \in \textup{Sg}^{\boldsymbol{T} \times \boldsymbol{T}\times \boldsymbol{T}}(\{ \langle x, x, y \rangle, \langle y , y, y\rangle, \langle x, y, x\rangle\}).
\]
\end{Proposition}

\begin{proof}
Clearly, $\langle x, x, x \rangle \in \textup{Sg}^{\boldsymbol{T} \times \boldsymbol{T}\times \boldsymbol{T}}(\{ \langle x, x, y \rangle, \langle y , y, y\rangle, \langle x, y, x\rangle\})$ iff there exists a term $\varphi(x, y, z)$ such that
\[
\varphi^{\boldsymbol{T}}(x, y, x) = \varphi^{\boldsymbol{T}}(x, y, y) = \varphi^{\boldsymbol{T}}(y, y, x) = x.
\]
As $\boldsymbol{T}_\mathsf{K}(x, y) = \boldsymbol{T}_{\VVV(\mathsf{K})}(x, y)$ (see, e.g., \cite[Lem.\ 2.1.13]{Go98a}), we can harmlessly replace $\boldsymbol{T}$ by $\boldsymbol{T}_{\VVV(\mathsf{K})}(x, y)$ in the above display. When phrased in this way, the display becomes equivalent to the demand that
\[
\VVV(\mathsf{K}) \vDash \varphi(x, y, x) \thickapprox \varphi(x, y, y) \thickapprox \varphi(y, y, x) \thickapprox x
\]
(see, e.g., \cite[Thm.\ II.11.4]{BuSa00}). Lastly, the existence of a term $\varphi$ satisfying the above condition is equivalent to the demand that $\VVV(\mathsf{K})$ is arithmetical by Theorem \ref{Thm : Pixley arithmetical}.
\end{proof}

We are now ready to prove Theorem \ref{Thm : weak ES implies CP}.

\begin{proof}[Proof of Theorem \ref{Thm : weak ES implies CP}]
Let $\mathsf{K}$ be a congruence distributive quasivariety for which $\mathsf{K}_{\textup{\textsc{rfsi}}}$ is closed under nontrivial subalgebras. Suppose, with a view to contradiction, that $\mathsf{K}$ has the weak ES property and that $\VVV(\mathsf{K})$ is not arithmetical. Then let $\boldsymbol{T} \coloneqq \boldsymbol{T}_\mathsf{K}(x, y)$ and define
\[
\A \coloneqq \textup{Sg}^{\boldsymbol{T} \times \boldsymbol{T}\times \boldsymbol{T}}(\{ \langle x, x, y \rangle, \langle y , y, y\rangle, \langle x, y, x\rangle\}) \, \, \text{ and } \, \, \B \coloneqq \textup{Sg}^{\boldsymbol{T} \times \boldsymbol{T} \times \boldsymbol{T}}(A \cup \{ \langle x, x, x \rangle\}).
\]
The definition of $\B$ guarantees that $\A \leq \B$ is almost total. Furthermore, the inclusion $\A \leq \B$ is proper because of Proposition \ref{Prop : congruence permutable} and the assumption that $\VVV(\mathsf{K})$ is not arithmetical. Therefore, Proposition \ref{Prop : full quotients} yields some $\theta \in \mathsf{Con}_\mathsf{K}(\B)$ such that $\A/\theta \leq \B/\theta$ is full in $\mathsf{K}$. Since $\A/\theta \leq \B/\theta$ is proper and $\B/\theta$ is generated by $A/\theta \cup \{ \langle x, x, x \rangle/\theta\}$, we have
\begin{equation}\label{Eq : weak ES implies CP 1}
\langle x, x, x \rangle / \theta \notin A / \theta.
\end{equation}

Now, recall that $\boldsymbol{T}$ is generated by $x$ and $y$. Since by construction $\B \leq \boldsymbol{T} \times \boldsymbol{T}\times \boldsymbol{T}$ and the triples
$\langle x, x, x \rangle$ and $\langle y, y, y \rangle$ belong to $\B$, the canonical projections $p_1, p_2, p_3 \colon \B \to \boldsymbol{T}$ are all surjective. Therefore, $\B \leq \boldsymbol{T} \times \boldsymbol{T} \times \boldsymbol{T}$ is a subdirect product. As $\mathsf{K}$ is congruence distributive by assumption, we can apply Theorem \ref{Thm : CD varieties} obtaining some $\theta_1, \theta_2, \theta_3 \in \mathsf{Con}_\mathsf{K}(\boldsymbol{T})$ such that $\theta = (\theta_1 \times \theta_2 \times \theta_3){\upharpoonright}_{B}$. 

Observe that
  \begin{align*}
\theta{\upharpoonright}_{A} &= (\theta_1 \times \theta_2\times \theta_3){\upharpoonright}_{A} = ((\theta_1 \times T^2\times T^2) \cap (T^2 \times \theta_2\times T^2) \cap (T^2 \times  T^2 \times\theta_3)){\upharpoonright}_{A} \\
&= (\theta_1 \times T^2\times T^2){\upharpoonright}_{A} \cap (T^2 \times \theta_2\times T^2){\upharpoonright}_{A} \cap  (T^2 \times T^2 \times\theta_3){\upharpoonright}_{A},
  \end{align*}
where the first equality holds because $\theta = (\theta_1 \times \theta_2 \times \theta_3){\upharpoonright}_{B}$ and $\A \leq \B$, the second because $\theta_1 \times \theta_2 \times \theta_3 = (\theta_1 \times T^2\times T^2) \cap (T^2 \times \theta_2\times T^2) \cap (T^2 \times  T^2 \times\theta_3)$, and the third is straightforward.

Recall that $\A / \theta \leq \B / \theta$ is full in $\mathsf{K}$. On the other hand, as $\mathsf{K}$ has the weak ES property by assumption, $\A / \theta \leq \B / \theta$ is not epic in $\mathsf{K}$. Therefore, we can apply Corollary \ref{Cor : A RFSI} obtaining that $\A / \theta \in \SSS(\mathsf{K}_{\textsc{rsi}}) \subseteq \SSS(\mathsf{K}_{\textsc{rfsi}})$. Since $\mathsf{K}_{\textsc{rfsi}}$ is closed under nontrivial subalgebras by assumption, $\A / \theta$ is trivial or RFSI. Thus, $\theta{\upharpoonright}_{A} = A \times A$ or $\theta{\upharpoonright}_{A} \in \mathsf{Irr}_{\mathsf{K}}(\A)$, and hence the above display implies that
\[
\theta{\upharpoonright}_{A} = (\theta_1 \times T^2\times T^2){\upharpoonright}_{A} \, \, \text{ or } \, \, \theta{\upharpoonright}_{A} =(T^2 \times \theta_2\times T^2){\upharpoonright}_{A} \, \, \text{ or } \, \, \theta{\upharpoonright}_{A} = (T^2 \times  T^2 \times\theta_3){\upharpoonright}_{A}.
\]
We show that $\langle x,y \rangle \in \theta_2 \cup \theta_3$ in any of these three cases.  If $\theta{\upharpoonright}_{A} = (\theta_1 \times T^2\times T^2){\upharpoonright}_{A}$, then $\langle x,x,y \rangle/\theta = \langle x,y,x \rangle/\theta$. If $\theta{\upharpoonright}_{A} = (T^2 \times \theta_2\times T^2){\upharpoonright}_{A}$, then $\langle y,y,y \rangle/\theta = \langle x,y,x \rangle/\theta$. If $\theta{\upharpoonright}_{A} = (T^2 \times  T^2 \times\theta_3){\upharpoonright}_{A}$, then $\langle x,x,y \rangle/\theta = \langle y,y,y \rangle/\theta$. Therefore, in all the three cases we have $\langle x,y \rangle \in \theta_2 \cup \theta_3$ because $\theta = (\theta_1 \times \theta_2 \times \theta_3){\upharpoonright}_{B}$.

We conclude the proof by showing that $\langle x,y \rangle \in \theta_2 \cup \theta_3$ implies $\langle x, x, x \rangle / \theta \in A / \theta$, a contradiction with~(\ref{Eq : weak ES implies CP 1}). If $\langle x,y \rangle \in \theta_2$, then $\langle x, x, x \rangle / \theta = \langle x, y, x \rangle / \theta\in A / \theta$. Otherwise, $\langle x,y \rangle \in \theta_3$, which yields $\langle x, x, x \rangle / \theta = \langle x, x, y \rangle / \theta\in A / \theta$. 
\end{proof}

\begin{Remark}
Theorem \ref{Thm : weak ES implies CP} cannot be strengthened by dropping any of the assumptions on $\mathsf{K}$: congruence distributivity, the closure of $\mathsf{K}_{\textsc{fsi}}$ under nontrivial subalgebras, and the weak ES property. For recall that the varieties of distributive lattices $\mathsf{DL}$ and semilattices $\mathsf{SL}$ are not congruence permutable. However,
\benroman
\item $\mathsf{SL}$ has the ES property and $\mathsf{SL}_{\textsc{fsi}}$  is closed under nontrivial subalgebras;
\item the variety of lattices is congruence distributive and has the ES property (see, e.g., \cite[p.\ 102]{KMPT83});
\item $\mathsf{DL}$ is congruence distributive and $\mathsf{DL}_{\textsc{fsi}}$  is closed under nontrivial subalgebras. \qed
\eroman
\end{Remark}

\begin{exa}\label{Cor : final corrolary}
A variety $\mathsf{K}$ is said to be \emph{filtral} \cite{Magari1969} when for every subdirect product $\A \leq \prod_{i \in I}\A_i$ where $\{ \A_i : i \in I \}$ is a family of SI members of $\mathsf{K}$ and every $\theta \in \mathsf{Con}(\A)$ there exists a filter $F$ over
$I$ such that
\[
\theta = \{ \langle a, b \rangle \in A \times A : \llbracket a = b \rrbracket \in F \},
\]
where $\llbracket a = b \rrbracket \coloneqq \{ i \in I : a(i) = b(i) \}$. 
Moreover, a variety $\mathsf{K}$ is a \emph{discriminator variety} \cite{We78} when $\mathsf{K} = \VVV(\mathsf{M})$ for some class of algebras $\mathsf{M}$ for which there exists a term $\varphi(x, y, z)$ such that for every $\A \in \mathsf{M}$ and $a, b, c \in A$,
\[
		\varphi^\A(a, b, c) \coloneqq \begin{cases}
			c & \text{ if } a = b;\\
			a & \text{ otherwise.}
		\end{cases}
		\]
Notably, discriminator varieties coincide with the congruence permutable filtral varieties \cite{BKP-EDPC-II,FK83}. 

We will show that every filtral variety $\mathsf{K}$ with the weak ES property is a discriminator variety (this can also be inferred from \cite{CamperVagg}). Indeed, by what we observed above, it suffices to verify that $\mathsf{K}$ is congruence permutable. But this is a straightforward consequence of Theorem \ref{Thm : weak ES implies CP} because every filtral variety is congruence distributive and its class of FSI members is closed under nontrivial subalgebras (see, e.g., \cite[Cor.\ 6.5(i, iv)]{CaRa17}).
\qed
\end{exa}

\paragraph{\bfseries Acknowledgements.}
The second author was supported by the ayuda PREP$2022$-$000927$ financiada por MICIU/AEI/$10$.$13039$/$501100011033$ y por FSE+ as part of the proyecto PID$2022$-$141529$NB-C$21$ de investigaci\'on financiado por MICIU/AEI/$10$.$13039$/$501100011033$ y por FEDER, UE.

The third author was supported by the proyecto PID$2022$-$141529$NB-C$21$ de investigaci\'on financiado por MICIU/AEI/$10$.$13039$/$501100011033$ y por FEDER, UE. He was also supported by the Research Group in Mathematical Logic, $2021$SGR$00348$ funded by the Agency for Management of University and Research Grants of the Government of Catalonia, as well as by the MSCA-RISE-Marie Skłodowska-Curie Research and Innovation Staff Exchange (RISE) project MOSAIC $101007627$ funded by Horizon $2020$ of the European Union.

\bibliographystyle{plain}

\end{document}